\documentclass[10pt,twoside]{article}
\usepackage{amsmath,amsthm}
\usepackage[latin1]{inputenc}
\usepackage{picinpar}
\usepackage{longtable}
\usepackage{amsrefs}
\usepackage[dvips]{graphicx}
\usepackage[mathscr]{eucal}
\usepackage{calc}
\usepackage{ifthen}
\usepackage{dcpic,pictexwd}
\usepackage[all]{xy}
\usepackage{enumerate}
\usepackage{amssymb,latexsym}
\usepackage{amsthm}
\usepackage[dvips]{graphics}
\usepackage{color}
\usepackage{tabularx}
\usepackage{makeidx}
\usepackage{tkz-graph}
\usepackage{tikz-cd}
\usetikzlibrary{calc}
\usepackage{tikz}
\usetikzlibrary{arrows}
\usepackage{stackrel}
\usepackage{multicol}
\usepackage{float}
\newtheorem{lema}{Lemma}
\newtheorem{teor}{Theorem}
\newtheorem{prop}{Proposition}
\newtheorem{corol}{Corollary}

\theoremstyle{definition}
\newtheorem{Nota}{Remark}
\newtheorem{exam}{Example}

\begin{document}
\thispagestyle{plain}
\par\bigskip
\begin{centering}

\textbf{Homological Ideals as Integer Specializations of Some Brauer Configuration Algebras}

\end{centering}\par\bigskip
\begin{centering}
\footnotesize{Pedro Fernando Fern\'andez Espinosa }\\
\footnotesize{Agust\'{\i}n Moreno Ca\~{n}adas}\\
\end{centering}
\par\bigskip
\small{In this paper homological ideals associated to some Nakayama algebras are characterized and enumerated via integer specializations of some suitable Brauer configuration algebras. Besides, it is shown how the number of such homological ideals can be connected with the categorification process of Fibonacci numbers defined by Ringel and Fahr. }
\par\bigskip
\small{\textit{Keywords and phrases}: Brauer configuration algebra,  categorification, homological ideal, integer specialization.}

\bigskip \small{Mathematics Subject Classification 2010 : 16G20; 16G60; 16G30.}

\section{Introduction}

Homological ideals or strong idempotent ideals are ideals of an algebra introduced by  Auslander, Platzeck and Todorov in \cite{Auslander}. These ideals arise from the research of  heredity ideals and quasi-hereditary algebras. For these ideals the corresponding quotient map induces a full and faithful functor between derived categories. Recently, homological ideals have been studied in different contexts, for instance Gatica, Lanzilotta and Platzeck and independently Xu and Xi established some relationships with the so called finitistic dimension conjecture and the Igusa-Todorov functions \cite{Gatica}. Furthermore, De la Pe\~na and Xi  in \cite{JAP} and Armenta in \cite{Armenta} studied the impact of these ideals in the context of Hochschild cohomology and one point extensions.
\par\bigskip
This work deals with the combinatorial properties of homological ideals associated to some path algebras and their relationships with the novel Brauer Configuration algebras which have been introduced recently by Green and Schroll in \cite{Green}. In particular, we introduce the notion of the message of a Brauer configuration, such messages enable to compute the number of homological ideals associated to some Nakayama algebras. Moreover, such number of ideals allow us to obtain an alternative version of the partition formula  for even-index Fibonacci numbers given by Ringel and Fahr in \cite{Fahr1} attaining in this way a new algebraic interpretation for these numbers. Worth noting that Ringel and Fahr devoted works \cite{Fahr1}, \cite{Fahr2}, and \cite{Fahr3} to this kind of interpretations also called categorifications. \par\bigskip

This paper is distributed as follows. In Section 2, we recall main notation and definitions regarding homological ideals and Brauer configuration algebras. In particular, we introduce the notion of integer specialization of a Brauer configuration and the concept of the message of a Brauer configuration. In Section 3, we give combinatorial conditions to determine whether an idempotent ideal associated to some Nakayama algebras is homological or not and it is reminded the notion of categorification in the sense of Ringel and Fahr. We also give the number of such ideals via the integer specialization of a suitable Brauer configuration algebra and its corresponding message. Moreover, we use the number of homological ideals to establish a partition formula for even-index Fibonacci numbers. Some interesting sequences in the On-line Encyclopedia of Integer Sequences (OEIS, \cite{OEIS}) arising from these computations are described as well.

\section{Preliminaries}
In this section, we recall main definitions and notation to be used throughout the paper \cite{Green, Auslander, Sierra, JAP,Armenta}.

\subsection{Homological Ideals}

For an algebra $A$ we mean a finite dimensional basic and connected algebra over an algebraically closed field $k$. We denote the category of finite dimensional right $A$-modules as $mod(A)$, and the bounded derived category of $mod(A)$ as $D^b(A)$. We will assume that $A$ is a bounded path algebra of the form $kQ/{I}$ with $Q$ a finite quiver and $I$ an admissible ideal.

\par\bigskip
An epimorphism of algebras $\phi:A \to B$ is called an \textit{homological epimorphism} if it induces a full and faithful functor
\[D^b(\phi^*) : D^b(B) \to D^b(A).\]

Let $I$ be a two sided ideal of $A$. Since the quotient map $\pi:A \to A/I$ is an epimorphism then the induced functor $\pi^*:mod(A/I)\to mod(A)$ is full and faithful.
\par\bigskip

A two sided ideal $I$ of $A$ is \textit{homological} if the quotient map $\pi:A\to A/I$ is an homological epimorphism.

\par\bigskip

The following results characterize homological ideals \cites{Auslander, JAP}.

\begin{prop}\label{JAP} Let $I$ be an ideal of $A$, then
\begin{enumerate}
\item $I$ is an homological ideal of $A$ if and only if $\mathrm{Tor}_n^A(I,A/I)=0$ for all $n\geq0$. In this case, $I$ is idempotent.

\item If $I$ is idempotent and $A$-projective, then $I$ is homological.

\item If $I$ is idempotent then $I$ is homological if and only if $\mathrm{Ext}^{n}_{A}(I,A/I)=0$ for all $n\geq0$.
\end{enumerate}

\end{prop}

We denote the \textit{trace} of an $A$-module $M$ in an $A$-module $N$ as
$$tr_M(N):={\displaystyle\sum_{f\in Hom_{A}(M,N)}} Im(f) \subset N.$$

\addtocounter{Nota}{1}
\begin{Nota}
We recall that according to Auslander et al. \cite{Auslander}, if $P$ is an $A$-projective module then $tr_P(A)$ is an idempotent ideal of $A$ and one obtains all the idempotent ideals of $A$ this way.  
\end{Nota}

\begin{Nota}
Note that, since the homological ideals are idempotent ideals and the idempotent ideals are traces of projective modules over $A$ then there is always a finite number of homological ideals.
\end{Nota}

Following the assumption that $A$ is a bounded quiver algebra of the form $kQ/I$ and the number of vertices of $Q$ are finite for every subset $\{a_1,...,a_m \} \subset Q_0$, we will assume the following notation for every idempotent ideal generated by the  trace of $P(a_1) \oplus \cdots \oplus P(a_m)$ in $A$:
\begin{equation}\label{HI1}
\begin{split}
I_{a_1,...,a_m} = tr_{\Big(P(a_1) \oplus \cdots \oplus P(a_m)\Big) }(A).
\end{split}
\end{equation}

In this paper, we combine  tools developed by Auslander et al. in \cite{Auslander}, Xi and De la Pe\~na in \cite{JAP} and the integer specializations of some Brauer configuration (see Section \ref{MBC}) to establish an explicit formula for the number of homological ideals associated to some Nakayama algebras. This number allows to establish a partition formula for even-index Fibonacci numbers as Ringel and Fahr define in \cite{Fahr1, Fahr2, Fahr3}.

\subsection{Brauer Configuration Algebras}

Brauer configuration algebras were introduced by Green and Schroll in \cite{Green} as a generalization of Brauer graph algebras which are biserial algebras of tame representation type and whose representation theory is encoded by some combinatorial data based on graphs. Actually, underlying every Brauer graph algebra is a finite graph with acyclic orientation of the edges at every vertex and a multiplicity function \cite{Green}.  The construction of a Brauer graph algebra is a special case of the construction of a Brauer configuration algebra in the sense that every  Brauer graph is a Brauer configuration with the restriction that every polygon is a set with two vertices. In the sequel, we give precise definitions of a Brauer configuration and a Brauer configuration algebra.\par\bigskip

A \textit{Brauer configuration} $\Gamma$ is a quadruple of the form $\Gamma=(\Gamma_{0},\Gamma_{1},\mu,\mathcal{O})$ where:

\begin{enumerate}
\item[(B1)] $\Gamma_{0}$ is a finite set whose elements are called \textit{vertices},
\item [(B2)]  $\Gamma_{1}$ is a finite collection of multisets called \textit{polygons}. In this case, if $V\in \Gamma_{1}$ then the elements of $V$ are vertices possibly with repetitions,  $\mathrm{occ}(\alpha,V)$ denotes the frequency of the vertex $\alpha$ in the polygon $V$ and the \textit{valency} of $\alpha$ denoted  $val(\alpha)$ is defined  in such a way that: \begin{equation}
\begin{split}
val(\alpha)&=\underset{V\in\Gamma_{1}}{\sum}\mathrm{occ}(\alpha,V).
\end{split}
\end{equation}
\item [(B3)]$\mu$ is an integer valued function such that $\mu:\Gamma_{0}\rightarrow \mathbb{N}$ where $\mathbb{N}$ denotes the set of positive integers, it is called the \textit{multiplicity function},
\item[(B4)] $\mathcal{O}$ denotes an orientation defined on $\Gamma_{1}$ which is a choice, for each vertex $\alpha \in \Gamma_0$, of a cyclic ordering of the polygons in which $\alpha$ occurs as a vertex, including repetitions, we denote $S_{\alpha}$ such collection of polygons.  More specifically, if $S_{\alpha}=\{V^{(\alpha_{1})}_{1},V^{(\alpha_{2})}_{2},\dots, V^{(\alpha_t)}_{t}\}$ is the collection of polygons where the vertex $\alpha$ occurs with $\alpha_{i}=\mathrm{occ}(\alpha,V_{i})$ and $V^{(\alpha_{i})}_{i}$ meaning that $S_{\alpha}$ has $\alpha_{i}$ copies of $V_{i}$ then an orientation $\mathcal{O}$ is obtained by endowing a linear order $\leq$ to $S_{\alpha}$  and adding a relation $V_{t}\leq V_{1}$, if $V_{1}=\mathrm{min}\hspace{0.1cm}S_{\alpha}$ and $V_{t}=\mathrm{max}\hspace{0.1cm}S_{\alpha}$,
\item [(B5)] Every vertex in $\Gamma_{0}$ is a vertex in at least one polygon in $\Gamma_{1}$,
\item[(B6)] Every polygon has at least two vertices,
\item[(B7)] Every polygon in $\Gamma_{1}$ has at least one vertex $\alpha$ such that $val(\alpha)\mu(\alpha)>1$.
\end{enumerate}

The set $(S_{\alpha},\leq)$ is called the \textit{successor sequence} at the vertex $\alpha$.\par\bigskip

A vertex $\alpha\in\Gamma_{0}$ is said to be \textit{truncated} if $val(\alpha)\mu(\alpha)=1$, that is, $\alpha$ is truncated if it occurs exactly once in exactly one $V\in\Gamma_{1}$ and $\mu(\alpha)=1$. A vertex is \textit{non-truncated} if it is not truncated.

\begin{center}
\textbf{The Quiver of a Brauer Configuration Algebra}
\end{center}

The quiver $Q_{\Gamma}=((Q_{\Gamma})_{0},(Q_{\Gamma})_{1})$ of a Brauer configuration algebra is defined in such a way that the vertex set $(Q_{\Gamma})_{0}=\{v_{1},v_{2},\dots,v_{m}\}$ of $Q_{\Gamma}$ is in correspondence with the set of polygons $\{V_{1},V_{2},\dots,V_{m}\}$ in $\Gamma_{1}$, noting that there is one vertex in $(Q_{\Gamma})_{0}$ for every polygon in $\Gamma_{1}$.\par\bigskip

Arrows in $Q_{\Gamma}$ are defined by the successor sequences. That is, there is an arrow $v_{i}\stackrel{s_{i}}{\longrightarrow}v_{i+1}\in (Q_{\Gamma})_{1}$ provided that $V_{i}\leq V_{i+1}$ in $(S_{\alpha},\leq)\cup\{V_{t}\leq V_{1}\}$ for some non-truncated vertex $\alpha\in\Gamma_{0}$. In other words, for each non-truncated vertex $\alpha\in\Gamma_{0}$ and each successor $V'$ of $V$ at $\alpha$, there is an arrow from $v$ to $v'$ in $Q_{\Gamma}$ where $v$ and $v'$ are the vertices in $Q_{\Gamma}$ associated to the polygons $V$ and $V'$ in $\Gamma_{1}$, respectively.\par\bigskip

\begin{centering}
\textbf{The Ideal of Relations and Definition of a Brauer Configuration Algebra}\par\bigskip
\end{centering}

Fix a polygon $V\in\Gamma_{1}$ and suppose that $\mathrm{occ}(\alpha,V)=t\geq1$ then there are $t$ indices
$i_{1},\dots, i_{t}$ such that $V=V_{i_{j}}$. Then the \textit{special $\alpha$-cycles} at $v$ are the cycles $C_{i_{1}}, C_{i_{2}},\dots, C_{i_{t}}$ where $v$ is the vertex in the quiver of $Q_{\Gamma}$ associated to the polygon $V$.
If $\alpha$ occurs only once in $V$ and $\mu(\alpha)=1$ then there is only one special $\alpha$-cycle at $v$.

\par\smallskip

Let $k$ be a field and $\Gamma$ a Brauer configuration. The \textit{Brauer configuration algebra associated to $\Gamma$} is defined to be the bounded path algebra $\Lambda_{\Gamma}=kQ_{\Gamma}/I_{\Gamma}$, where $Q_{\Gamma}$ is the quiver associated to $\Gamma$ and $I_{\Gamma}$ is the \textit{ideal} in $kQ_{\Gamma}$ generated by the following set of relations $\rho_{\Gamma}$ of type I, II and III.\par\bigskip
\begin{enumerate}
\item \textbf{Relations of type I}. For each polygon $V=\{\alpha_{1},\dots, \alpha_{m}\}\in \Gamma_{1}$ and each pair of non-truncated vertices $\alpha_{i}$ and $\alpha_{j}$ in $V,$ the set of relations $\rho_{\Gamma}$ contains all relations of the form $C^{\mu(\alpha_{i})}-C'^{\mu(\alpha_{j})}$ where $C$ is a special $\alpha_{i}$-cycle
and $C'$ is a special $\alpha_{j}$-cycle.

\item \textbf{Relations of type II}. Relations of type II are all paths of the form $C^{\mu(\alpha)}a$ where $C$ is a special $\alpha$-cycle and $a$ is the first arrow in $C$.

\item \textbf{Relations of type III}. These relations are quadratic monomial relations of the form $ab$ in $kQ_{\Gamma}$ where $ab$ is not a subpath of any special cycle unless $a=b$ and $a$ is a loop associated to a vertex of valency 1 and $\mu(\alpha)>1$.

\end{enumerate}

As an example for $n \geq 4$ fixed, we consider a Brauer configuration $\Gamma_n=(\Gamma_{0},\Gamma_{1},\mu,\mathcal{O})$ such that:
\begin{enumerate}\label{BrauerH}
\item $\Gamma_{0}=\{n-k-1\in\mathbb{N}\mid 2 \leq k \leq n-1\}\cup\{n-2\}$,
\item $\Gamma_{1}=\{U_k=\{n-2,n-k-1\}\mid 2\leq k\leq n-1\}$.
\item The orientation $\mathcal{O}$ is defined in such a way that
\begin{enumerate}[$(a)$]
\item  Vertex $n-2$ has associated the successor sequence  $U_2<U_3<\cdots<U_{n-1}$, in this case, $val(n-2)=n-2$,
\item If $2\leq k\leq n-1$ then at vertex $n-k-1$, it holds that the corresponding successor sequence consists only of $U_k$, and for each $k$, $val(n-k-1)=1$.
\end{enumerate}
\item $\mu(n-2)=1$,
\item $\mu(n-k-1)=n-2$,\quad $2\leq k\leq n-1$.

\end{enumerate}

The ideal $I_{\Gamma_{n}}$ of the corresponding Brauer configuration algebra $\Lambda_{\Gamma_n}$ is generated by the following relations (see Figure \ref{examplebca}), for which it is assumed the following notation for the special cycles:
\begin{equation}\label{special}
\begin{split}
C^{U_k}_{n-2}&=\begin{cases}
a_1^{n-2}a_{2}^{n-2}\cdots a_{k-1}^{n-2}, &\quad\text{if } k=2,\\
a_{k-1}^{n-2}a_{k}^{n-2}\cdots a_{k-2}^{n-2},&\quad\text{otherwise}, \\
\end{cases}\\
C^{U_k}_{n-k-1}&=a_{1}^{n-k-1}.\\
\end{split}
\end{equation}

\begin{enumerate}
%Tipo 3
\item $a^{h}_{i}a^{s}_{r}$, if $h\neq s$, for all possible values of $i$ and $r$ unless for the loops associated to the vertices $n-k-1$,

% Tipo 1
\item $C^{U_k}_{n-2}-\left(C^{U_k}_{n-k-1}\right)^{n-2}$,\quad for all possible values of $k$,

%%% Tipo 2
\item $C^{U_k}_{n-2} a$  with $a$  being the first arrow of $C^{U_k}_{n-2}$  for all $k$,
\item $\left(C^{U_k}_{n-k-1}\right)^{n-2}a'$ with $a'$  being the first arrow of $C^{U_k}_{n-k-1}$ for all $k$.

\end{enumerate}
\par\bigskip
Figure \ref{examplebca} shows the quiver $Q_{\Gamma_n}$ associated to this configuration.

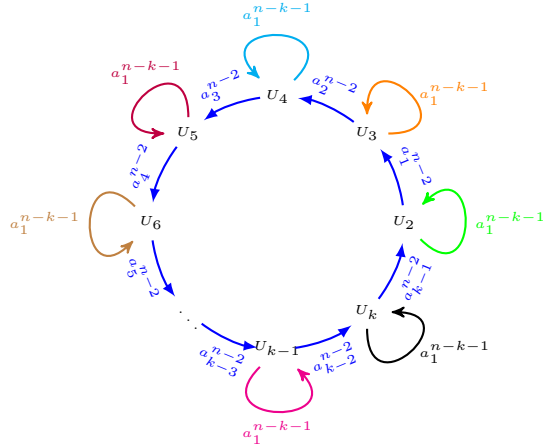
\begin{figure}[H]
\begin{center}
\begin{tikzpicture}[->,>=stealth',shorten >=1pt,thick,scale=0.56]
\def \radius {3cm}
\def \margin {8} % margin in angles, depends on the radius
{
\node at ({360/8 * (1 - 1)}:\radius)(U2) {\tiny{$U_2$}};

\node at ({360/8 * (2 - 1)}:\radius)(U3) {\tiny{$U_3$}};

\node at ({360/8 * (3 - 1)}:\radius) (U4){\tiny{$U_4$}};

\node at ({360/8 * (4 - 1)}:\radius)(U5) {\tiny{$U_5$}};

\node at ({360/8 * (5 - 1)}:\radius)(U6) {\tiny{$U_6$}};

\node at ({360/8 * (6 - 1)}:\radius)(U7) {\tiny{$\ddots$}};

\node at ({360/8 * (7 - 1)}:\radius)(Uk-1) {\tiny{$U_{k-1}$}};

\node at ({360/8 * (8 - 1)}:\radius)(Uk) {\tiny{$U_k$}};

%Arrows in the circle

\draw[->, >=latex,blue] ({360/8* (1 - 1)+\margin}:\radius)
arc ({360/8 * (1 - 1)+\margin}:{360/8 * (1)-\margin}:\radius)node[midway,sloped,above] {\tiny$a_1^{n-2}$};

\draw[->, >=latex,blue] ({360/8* (2 - 1)+\margin}:\radius)
arc ({360/8 * (2 - 1)+\margin}:{360/8 * (2)-\margin}:\radius)node[midway,sloped,above] {\tiny$a_2^{n-2}$};

\draw[->, >=latex,blue] ({360/8* (3 - 1)+\margin}:\radius)
arc ({360/8 * (3 - 1)+\margin}:{360/8 * (3)-\margin}:\radius)node[midway,sloped,above] {\tiny$a_3^{n-2}$};

\draw[->, >=latex,blue] ({360/8* (4 - 1)+\margin}:\radius)
arc ({360/8 * (4 - 1)+\margin}:{360/8 * (4)-\margin}:\radius)node[midway,sloped,above] {\tiny$a_4^{n-2}$};

\draw[->, >=latex,blue] ({360/8* (5 - 1)+\margin}:\radius)
arc ({360/8 * (5 - 1)+\margin}:{360/8 * (5)-\margin}:\radius)node[midway,sloped,below] {\tiny$a_5^{n-2}$};

\draw[->, >=latex,blue] ({360/8* (6 - 1)+\margin}:\radius)
arc ({360/8 * (6 - 1)+\margin}:{360/8 * (6)-\margin}:\radius)node[midway,sloped,below] {\tiny$a_{k-3}^{n-2}$};

\draw[->, >=latex,blue] ({360/8* (7 - 1)+\margin}:\radius)
arc ({360/8 * (7 - 1)+\margin}:{360/8 * (7)-\margin}:\radius)node[midway,sloped,below] {\tiny$a_{k-2}^{n-2}$};

\draw[->, >=latex,blue] ({360/8* (8 - 1)+\margin}:\radius)
arc ({360/8 * (8 - 1)+\margin}:{360/8 * (8)-\margin}:\radius)node[midway,sloped,below] {\tiny$a_{k-1}^{n-2}$};
}

%loops
\Loop[dist=2cm,dir=EA,style={green},label={\tiny$a_1^{n-k-1}$},labelstyle=right](U2)
\Loop[dist=2cm,dir=NOEA,style={orange},label={\tiny$a_1^{n-k-1}$},labelstyle=right](U3)
\Loop[dist=2cm,dir=NO,style={cyan},label={\tiny$a_1^{n-k-1}$},labelstyle=above](U4)
\Loop[dist=2cm,dir=NOWE,style={purple},label={\tiny$a_1^{n-k-1}$},labelstyle=above](U5)
\Loop[dist=2cm,dir=WE,style={brown},label={\tiny$a_1^{n-k-1}$},labelstyle=left](U6)
\Loop[dist=2cm,dir=SO,style={magenta},label={\tiny$a_1^{n-k-1}$},labelstyle=below](Uk-1)
\Loop[dist=2cm,dir=SOEA,style={black},label={\tiny$a_1^{n-k-1}$},labelstyle=right](Uk)
\end{tikzpicture}
\caption{The quiver $Q_{\Gamma_{n}}$ defined by the Brauer configuration $\Gamma_{n}$.}
\label{examplebca}
\end{center}
\end{figure}

The following results describe the structure of a Brauer configuration algebra \cite{Green}.

\addtocounter{teor}{3}
\begin{teor}\label{multiserial}
\textit{Let $\Lambda$ be a Brauer configuration algebra with Brauer configuration $\Gamma$}.
\begin{enumerate}
\item \textit{There is a bijective correspondence between the set of projective indecomposable $\Lambda$-modules and the polygons in $\Gamma$}.
\item \textit{If $P$ is a projective indecomposable $\Lambda$-module corresponding to a polygon $V$ in $\Gamma$. Then $\mathrm{rad}\hspace{0.1cm}P$ is a sum of $r$ indecomposable uniserial modules, where $r$ is the number
of (non-truncated) vertices of $V$ and where the intersection of any two of the uniserial modules is a simple $\Lambda$-module}.
\item \textit{A Brauer configuration algebra is a multiserial algebra}.
\item \textit{The number of summands in the heart of an indecomposable projective $\Lambda$-module $P$ such that $\mathrm{rad}^{2}\hspace{0.1cm}P\neq 0$ equals the number of non-truncated vertices of the polygons in $\Gamma$ corresponding to $P$ counting repetitions}.
\item \textit{If $\Lambda'$ is a Brauer configuration algebra obtained from $\Lambda$ by removing a truncated vertex of a polygon in $\Gamma_{1}$ with $d\geq 3$ vertices then $\Lambda$ is isomorphic to $\Lambda'$.}

\end{enumerate}

\end{teor}

\addtocounter{prop}{3}
\begin{prop}\label{dimension}
\textit{Let $\Lambda$ be a Brauer configuration algebra associated to the Brauer configuration $\Lambda$ and let $\mathcal{C}=\{C_{1},\dots, C_{t}\}$ be a full set of equivalence class representatives of special cycles. Assume that for $i=1,\dots,t$, $C_{i}$ is a special $\alpha_{i}$-cycle where $\alpha_{i}$ is a non-truncated vertex in $\Gamma$}. \textit{Then}
\begin{center}
$\mathrm{dim}_{k}\hspace{0.1cm}\Lambda=2|Q_{0}|+\underset{C_{i}\in\mathcal{C}}{\sum}|C_{i}|(n_{i}|C_{i}|-1)$,

\end{center}
\textit{where $|Q_{0}|$ denotes the number of vertices of $Q$, $|C_{i}|$ denotes the number of arrows in the $\alpha_{i}$-cycle $C_{i}$ and $n_{i}=\mu(\alpha_{i})$}.
\end{prop}

\begin{prop}\label{grading}
\textit{Let $\Lambda$ be the Brauer configuration algebra associated to a connected Brauer configuration $\Gamma$. The algebra $\Lambda$ has a length grading induced from the path algebra $kQ$ if and only if there is an $N\in\mathbb{Z}_{>0}$ such that for each non-truncated vertex $\alpha\in\Gamma_{0}$ $val(\alpha)\mu(\alpha)=N$.}
\end{prop}
Sierra \cite{Sierra} proved the following result regarding the center of a Brauer configuration algebra.

\addtocounter{teor}{2}
\begin{teor}\label{Serra}
\textit{Let $\Gamma$ be a reduced (i.e, without truncated vertices) and connected Brauer configuration and let $Q$ be its induced quiver and let $\Lambda$ be the induced Brauer configuration algebra such that $\mathrm{rad}^{2}\hspace{0.1cm}\Lambda \neq 0$ then the dimension of the center of $\Lambda$ denoted $\mathrm{dim}_{k}\hspace{0.1cm}Z(\Lambda)$ is given by the formula}:

\begin{equation}\label{Sierra}
\begin{split}
\mathrm{dim}_{k}\hspace{0.1cm}Z(\Lambda)&=1+\underset{\alpha\in\Gamma_{0}}{\sum}\mu(\alpha)+|\Gamma_{1}|-|\Gamma_{0}|+\#(Loops\hspace{0.1cm} Q)-|\mathscr{C}_{\Gamma}|.
\end{split}
\end{equation}

\textit{where} $|\mathscr{C}_{\Gamma}|=\{\alpha\in\Gamma_{0}\mid val(\alpha)=1, \hspace{0.1cm}and\hspace{0.1cm} \mu(\alpha)>1\}$.

\end{teor}

As an example the following is the numerology associated to the algebra $\Lambda_{\Gamma_{n}}=kQ_{\Gamma_n}/I_{\Gamma_n}$ with $Q_{\Gamma_n}$ as shown in Figure \ref{examplebca} and special cycles given in (\ref{special}), ($|r(Q_{\Gamma_n})|$ is the number of indecomposable projective modules. Note that, $|C_{i}|=val(i)$):

\begin{equation*}\label{theexample}
\begin{split}
|r(Q_{\Gamma_n})|&=n-2,\\
|C_{n-2}|&=n-2, \quad|C_{n-k-1}|=1,\\
\underset{\alpha\in\Gamma_{0}}{\sum}\underset{X\in\Gamma_{1}}{\sum}\mathrm{occ}(\alpha, X)&= n-1,\quad\text{the number of special cycles},\\
\mathrm{dim}_{k}\hspace{0.1cm}\Lambda_{\Gamma_n}&=2(n-2)+(n-2)(n-3)+(n-3)(n-2)=2(n-2)^2,\\
\mathrm{dim}_{k}\hspace{0.1cm}Z(\Lambda_{\Gamma_{n}})&=1+1+(n-2)^2+(n-2)-(n-1)+(n-2)-(n-2)=\\
&= n^{2}-4n+5.\\
\end{split}
\end{equation*}
\addtocounter{Nota}{4}

\begin{Nota}
$\Lambda_{\Gamma_{n}}$ is a Brauer graph algebra and according to Proposition \ref{grading}, the Brauer configuration algebra $\Lambda_{\Gamma_{n}}$ with quiver $Q_{\Gamma_{n}}$ shown in Figure \ref{examplebca} has a length grading induced by the path algebra $kQ_{\Gamma_{n}}$, provided that for any $\alpha\in \Gamma_{0}$ it holds that $\mu(\alpha)val(\alpha)=n-2$.
\end{Nota}

\subsection{Message of a Brauer Configuration}\label{MBC}
The concept of the message of a Brauer configuration is helpful to categorify some integer sequences in the sense of Ringel and Fahr (see Section 3.1 of the present document, \cite{Fahr1, Fahr2}).\par\smallskip

Let $\Gamma=\{\Gamma_{0},\Gamma_{1},\mu,\mathcal{O}\}$ be a Brauer configuration and let  $U\in\Gamma_{1}$ be a polygon such that $U=\{\alpha^{f_{1}}_{1},\alpha^{f_{2}}_{2},\dots,\alpha^{f_{n}}_{n}\}$, where $f_{i}=\mathrm{occ}(\alpha_{i},U)$. The term
\begin{equation}\label{word}
\begin{split}
w(U)&=\alpha^{f_{1}}_{1}\alpha^{f_{2}}_{2}\dots\alpha^{f_{n}}_{n}
\end{split}
\end{equation}
is said to be the \textit{word associated to $U$}. The sum
\begin{equation}\label{message}
\begin{split}
M(\Gamma)&=\underset{U\in\Gamma_{1}}{\sum}w(U)
\end{split}
\end{equation}
is said to be the \textit{message of the Brauer configuration $\Gamma$}.

\par\bigskip

An \textit{integer specialization} of a Brauer configuration $\Gamma$ is a Brauer configuration $\Gamma^{e}=(\Gamma^{e}_{0},\Gamma^{e}_{1},\mu^{e},\mathcal{O}^{e})$ endowed with a preserving orientation map $e:\Gamma_{0}\rightarrow \mathbb{N}$, such that
\begin{equation}\label{lspecialization}
\begin{split}
\Gamma^{e}_{0}&=\mathrm{Img}\hspace{0.1cm}e\subset\mathbb{N},\\
\Gamma^{e}_{1}&=e(\Gamma_{1}),\quad \text{if}\hspace{0.1cm}H\in \Gamma_{1}\hspace{0.1cm}\text{then}\hspace{0.1cm}e(H)=\{e(\alpha_{i})\mid \alpha_{i}\in H\}\in e(\Gamma_{1}), \\
\mu^{e}(e(\alpha))&=\mu(\alpha),\hspace{0.1cm}\text{for any}\hspace{0.1cm}\alpha\in\Gamma_{0}.
\end{split}
\end{equation}

Besides $e(U)\preceq e(V)$ in $\Gamma^{e}_{1}$ provided that $U\preceq V$ in $\Gamma_{1}$.\par\bigskip

We let $w^{e}(U)=(e(\alpha_{1}))^{f_{1}}(e(\alpha_{2}))^{f_{2}}\dots(e(\alpha_{n}))^{f_{n}}$ denote the specialization under $e$ of a word $w(U)$. In such a case, $M(\Gamma^{e})=\underset{U\in\Gamma^{e}_{1}}{\sum}w^{e}(U)$ is the \textit{specialized message} of the Brauer configuration $\Gamma$ with the usual integer sum and product (in general with the sum and product associated to $\mathrm{Img}\hspace{0.1cm}e$).

\addtocounter{exam}{8}
\begin{exam}\label{integerspecialization}
For the Brauer configuration $\Gamma_n$ whose associated quiver is shown in Figure 1, we define the specialization $e(\alpha)=2^{\alpha}$, $\alpha\in\Gamma_{0}$ with the concatenation in each word given by the difference of the specializations of the vertices belonging to a determined polygon, in such a case for $n$ fixed, we have:
\begin{equation}
\begin{split}
w(U_k)&=(n-2)(n-k-1), \text{ for } 2\leq k \leq n-1,\\
w^e(U_k)&=2^{n-2}-2^{n-k-1},\text{ for } 2\leq k \leq n-1,\\
M(\Gamma_n^e)&=\underset{U_k\in\Gamma_{1}}{\sum}w^{e}(U_k)=\sum_{k=1}^{n-1} 2^{n-2}-2^{n-k-1}.
\end{split}
\end{equation}  
\end{exam}

\section{Homological Ideals Associated to Nakayama Algebras}
In this section, we prove some combinatorial conditions which allow to establish whether an idempotent ideal in some Nakayama algebras is homological or not. We also give the number of homological ideals associated to these algebras via the integer specialization of the Brauer configuration $\Gamma_n$ defined in Example \ref{integerspecialization}. Moreover, we use the number of homological ideals to establish a partition formula for even-index Fibonacci numbers.\par\smallskip

Let $Q$ be either a linearly oriented quiver with underlying graph $\mathbb{A}_n$ or a cycle $\widetilde{\mathbb{A}_n}$ with cyclic
orientation. That is, $Q$ is one of the following quivers

\begin{figure}[H]
\begin{center}
\begin{tikzpicture}[scale=0.55]
\def \radius {2cm}
\def \margin {14} % margin in angles, depends on the radius
{
\node at ({360/8 * (1 - 1)}:\radius)(U2) {\tiny{$1$}};

\node at ({360/8 * (2 - 1)}:\radius)(U3) {\tiny{$2$}};

\node at ({360/8 * (3 - 1)}:\radius) (U4){\tiny{$3$}};

\node at ({360/8 * (4 - 1)}:\radius)(U5) {\tiny{$4$}};

\node at ({360/8 * (5 - 1)}:\radius)(U6) {\tiny{$5$}};

\node at ({360/8 * (6 - 1)}:\radius)(U7) {\tiny{$\ddots$}};

\node at ({360/8 * (7 - 1)}:\radius)(Uk-1) {\tiny{$n-1$}};

\node at ({360/8 * (8 - 1)}:\radius)(Uk) {\tiny{$n$}};

%Arrows in the circle

\draw[->, >=latex] ({360/8* (1 - 1)+\margin}:\radius)
arc ({360/8 * (1 - 1)+\margin}:{360/8 * (1)-\margin}:\radius)node[midway,sloped,above] {};

\draw[->, >=latex] ({360/8* (2 - 1)+\margin}:\radius)
arc ({360/8 * (2 - 1)+\margin}:{360/8 * (2)-\margin}:\radius)node[midway,sloped,above] {};

\draw[->, >=latex] ({360/8* (3 - 1)+\margin}:\radius)
arc ({360/8 * (3 - 1)+\margin}:{360/8 * (3)-\margin}:\radius)node[midway,sloped,above] {};

\draw[->, >=latex] ({360/8* (4 - 1)+\margin}:\radius)
arc ({360/8 * (4 - 1)+\margin}:{360/8 * (4)-\margin}:\radius)node[midway,sloped,above] {};

\draw[->, >=latex] ({360/8* (5 - 1)+\margin}:\radius)
arc ({360/8 * (5 - 1)+\margin}:{360/8 * (5)-\margin}:\radius)node[midway,sloped,below] {};

\draw[->, >=latex] ({360/8* (6 - 1)+\margin}:\radius)
arc ({360/8 * (6 - 1)+\margin}:{360/8 * (6)-\margin}:\radius)node[midway,sloped,below] {};

\draw[->, >=latex] ({360/8* (7 - 1)+\margin}:\radius)
arc ({360/8 * (7 - 1)+\margin}:{360/8 * (7)-\margin}:\radius)node[midway,sloped,below] {};

\draw[->, >=latex] ({360/8* (8 - 1)+\margin}:\radius)
arc ({360/8 * (8 - 1)+\margin}:{360/8 * (8)-\margin}:\radius)node[midway,sloped,below] {};
}

%% Linear
\node["\tiny{$1$}"] at (4,0)(1){$\bullet$};
\node["\tiny{$2$}"] at (5.5,0)(2){$\bullet$};

\node at (6.5,0)(dot){\tiny{$\cdots$}};

\node["\tiny{$n-1$}"] at (7.5,0)(n-1){$\bullet$};
\node["\tiny{$n$}"] at (9,0)(n){$\bullet$};

\node at (3,0)(or){\small{or}};

\path[->] (1) edge (2);
\path[->] (n-1) edge (n);

\end{tikzpicture}
\caption{Quiver $\widetilde{\mathbb{A}_n}$ with cyclic orientation and Dynkin diagram $\mathbb{A}_n$ linearly oriented.}
\end{center}
\end{figure}
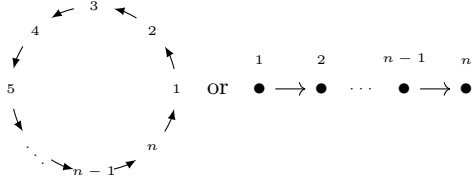\label{Nakayama}

A quotient $A$ of $kQ$ by an admissible ideal $I$ is called a \textit{Nakayama algebra} \cite{Happel}.  
\par\smallskip

In this work, for $n\geq 3$ fixed, we consider the algebras $A_{R_{(i, j, k)}}=kQ/I$ where $Q$ is a Dynkin diagram of type $\mathbb{A}_{n}$ linearly oriented and $I$ is an admissible ideal generated by one relation $R_{(i, j, k)}$ of length $k$ starting at a vertex $i$ and ending at a vertex $j$ of the given quiver, $1\leq i<j \leq n$. The following picture shows the general structure of  quivers $Q$ which we are focused in this paper.
\[  \mathbb{A}_n=1 \to \cdots \to \color{red}i \to i+1  \to \cdots \to i+k=j \color{black}\to j+1 \to\cdots \to n-1 \to n.
\]
The following Lemmas \ref{homological}-\ref{last} allow to determine which idempotent ideals of an algebra $A_{R_{(i, j, k)}}$ are also homological ideals. In this case, Lemmas \ref{homological} and \ref{center} regard the case whenever the idempotent ideal is generated by the trace of just one projective module associated to a vertex of the quiver.  
\addtocounter{lema}{9}
\begin{lema}\label{homological}
Every idempotent ideal $I_r$ of an algebra $A_{R_{(i, j, k)}}$ (see $\mathrm{(\ref{HI1}})$) with $j\leq r$ or $r \leq i$  is homological.
\end{lema}

\begin{proof}
For $r\leq i$, we have the following cases:

\begin{enumerate}
\item $tr_{P(r)}(P(t))=0$  if  $t>r$.

\item $tr_{P(r)}(P(t))=P(r)$ if  $t \leq r$, where $P(r)$ denotes the $r$-th projective module.
\end{enumerate}

If $r\geq j$, we consider the following cases:

\begin{enumerate}
\item $tr_{P(r)}(P(t))=P(r)$ if $i<t\leq r$, where $P(r)$ denotes the $r$-th projective module.
\item $tr_{P(r)}(P(t))=0$.

\end{enumerate}

In all cases $tr_{P(r)}(A_{R_{(i, j, k)}})=P(r)^{l}$ for some $l\in \mathbb{N}$. The result follows as a consequence of Proposition \ref{JAP}, item 2. We are done.
\end{proof}

\begin{lema}\label{center}
Every idempotent ideal $I_t$  of an algebra $A_{R_{(i, j, k)}}$, with $i+1\leq t \leq j-1$ is not homological.
\end{lema}

\begin{proof}
Consider $L_t=tr_{P(t)}P(i)=P(i)/S(i)\oplus\cdots \oplus S(t-1) $ where $S(k)$ denote the $k$-th simple module, also note that there are not morphisms from  $P(t)$ to $P(j)$ if $t\neq j$ which means that $\textrm{Ext}^{1}_{A_{R_{(i, j, k)}}}(L_t,P(j))$ is a direct summand of $\textrm{Ext}^{1}_{A_{R_{(i, j, k)}}}(I_t,A_{R_{(i, j, k)}}/I_t)$, provided that $L_t$ is a direct summand of $I_t$ and $P(j)$ is a direct summand of $A_{R_{(i, j, k)}}/I_{t}$. Applying the functor $\textrm{Hom}_{A_{R_{(i, j, k)}}}(-,P(j))$ to a projective resolution of $L_t$ with the form

\[
0 \to P(j) \to P(t) \to L_t \to 0
\]

it is obtained the sequence

\[
0 \to \textrm{Hom}_{A_{R_{(i, j, k)}}}(P(t),P(j)) \to  \textrm{Hom}_{A_{R_{(i, j, k)}}}(P(j),P(j)) \to 0.
\]

Thus, $\textrm{Ext}^{1}_{A_{R_{(i, j, k)}}}(L_t,P(j)) \cong k$ and $\textrm{Ext}^{1}_{A_{R_{(i, j, k)}}}(I_{i},A_{R_{(i, j, k)}}/I_i) \not= 0$. Then the idempotent ideal $I_t$ is not an homological ideal
as a consequence of Proposition \ref{JAP}, item 3.
\end{proof}

\begin{lema}\label{center2}

If each idempotent ideal $I_{\alpha_w}$  of an algebra $A_{R_{(i, j, k)}}$ is not homological then every idempotent ideal of the form $I_{\alpha_1,\ldots,\alpha_l}$ is not homological for $2\leq l\leq k-1$.    
\end{lema}

\begin{proof}
For $l$ fixed, we start by computing $I_{\alpha_1,\ldots,\alpha_l}$,  

\[I_{\alpha_1,\ldots,\alpha_l}=tr_{P(\alpha_1)\oplus \cdots \oplus P(\alpha_l)}(A_{R_{(i, j, k)}})= \displaystyle \sum_{w=1}^{l}tr_{P(\alpha_w)}(A_{R_{(i, j, k)}}) \]

In accordance with the hypothesis $\alpha_w \in [i+1,j-1]$ and taking into account that

\begin{equation}\label{middle}
\begin{split}
tr_{P(\alpha_w)}(A_{R_{(i, j, k)}})=\underbrace{L_{\alpha_w}}_{i- times} \oplus  \underbrace{ P(\alpha_w)}_{\alpha_w-i-times} \oplus\underbrace{0}_{n-\alpha_w-times}
\end{split}
\end{equation}

\begin{equation}\label{total}
\begin{split}
tr_{P(\alpha_1)\oplus \cdots \oplus P(\alpha_l)}(A_{R_{(i, j, k)}})=\underbrace{L_{\alpha_1}}_{i- times}\oplus  \bigoplus_{w=1}^{l} P(\alpha_w) \oplus\underbrace{0}_{n-i-l-times}
\end{split}
\end{equation}

it holds that according to the identity (\ref{total}), $P(j)$ is a direct summand of  $A_{R_{(i, j, k)}}/I_{\alpha_1\ldots\alpha_l}$ and $L_{\alpha_1}$ has the following projective resolution

\[
0 \to P(j) \to P(\alpha_1) \to L_{\alpha_1} \to 0
\]

Applying the functor $ \textrm{Hom}_{A_{R_{(i, j, k)}}}(-,P(j))$,  we have that $\textrm{Ext}^{1}_{A_{R_{(i, j, k)}}}(L_{\alpha_1},P(j))\neq 0$ and by Proposition \ref{JAP} item 3, we conclude that the idempotent ideal $I_{\alpha_1\ldots\alpha_l}$ is not an homological ideal.
\end{proof}

\begin{lema}\label{17}
For $l$ fixed, if each idempotent ideal $I_{\alpha_w}$  of an algebra $A_{R_{(i, j, k)}}$ with $1 \leq w \leq l$ is homological then every idempotent ideal of the form $I_{\alpha_1,\ldots,\alpha_l}$ is also homological.  
\end{lema}

\begin{proof}
It suffices to observe that $tr_{P(\alpha_w)}(A_{R_{(i, j, k)}})=P(\alpha_w)^{l}$ for some $l\in \mathbb{N}$.
\end{proof}

\begin{lema}\label{18}
Every ideal $I_{i, t}$ or $I_{t, j}$ of an algebra $A_{R_{(i, j, k)}}$ is homological.
\end{lema}

\begin{proof}
In accordance with the previous Lemma we can conclude that if $I_t$ is homological then the result holds. If it is not the case then we consider the following cases:

\begin{enumerate}
\item For $I_{t}$ non homological we can compute  $I_{i,t}= tr_{P(i)\oplus P(t)}(A_{R_{(i, j, k)}})$ (see identity (\ref{middle})) since  $r\leq i$ then $tr_{P(i)}P(r)=P(i)$ therefore ideal  $I_{i,t}$ is projective and idempotent. Thus for Proposition \ref{JAP}, item 2. We conclude that ideal  $I_{i,t}$ is homological.
\item  We start by computing $I_{t,j}$ as follows:

\[I_{t,j}=tr_{P(t)\oplus P(j)}(A_{R_{(i, j, k)}})=\underbrace{L_{t}}_{i- times}\oplus \underbrace{ P(t)}_{t-i-times}\oplus \underbrace{P(j)}_{j-t-times} \oplus\underbrace{0}_{n-j-times}  \]

 $A_{R_{(i, j, k)}}/I_{t, j}$ is given by:

\[A_{R_{(i, j, k)}}/I_{t, j}= \frac{P(1)\oplus P(2) \oplus \cdots \oplus P(i)\oplus \cdots \oplus P(t) \oplus \cdots \oplus P(j)\oplus \cdots \oplus P(n)}{L_t \oplus \cdots \oplus L_t\oplus P(t) \oplus \cdots  \oplus P(t)\oplus P(j) \oplus \cdots \oplus P(j)\oplus 0 \oplus \cdots \oplus 0}\]

In order to compute $\textrm{Ext}^{1}_{A_{R_{(i, j, k)}}}=(I_{t, j},A_{R_{(i,j,k)}}/I_{t, j})$ we consider the projective resolution of $L_t$

\[
0 \to P(j) \to P(t) \to L_t \to 0.
\]

Applying the functor $ \textrm{Hom}_{A_{R_{(i, j, k)}}}(-,A_{R_{(i, j, k)}}/I_{t, j})$ we obtain:

\[
0 \to  \textrm{Hom}_{A_{R_{(i, j, k)}}}(P(t),A_{R_{(i, j, k)}}/I_{t, j}) \to \textrm{Hom}_{A_{R_{(i, j, k)}}}(P(j),A_{R_{(i, j, k)}}/I_{t, j}) \to 0
\]

Taking into account that \par\smallskip
\resizebox{11.5cm}{!} {
$ \left\{ \begin{array}{l}
 \textrm{Hom}_{A_{R_{(i, j, k)}}}(P(t),\frac{P(z)}{L_t})=0 \quad if \quad  1 \leq z \leq i   \\
\\ \textrm{Hom}_{A_{R_{(i, j, k)}}}(P(t),\frac{P(y)}{P(t)})=0  \quad if \quad   i+1 \leq y \leq t-1\\
\\  \textrm{Hom}_{A_{R_{(i, j, k)}}}(P(t),\frac{P(v)}{P(j)})=0  \quad if \quad  t+1 \leq v \leq j-1 \\
\\ \textrm{Hom}_{A_{R_{(i, j, k)}}}(P(t),P(u))=0  \quad if \quad   j+1 \leq u \leq n\\
\end{array}
\right.$
$ \left\{ \begin{array}{l}
 \textrm{Hom}_{A_{R_{(i, j, k)}}}(P(j),\frac{P(z)}{L_t})=0 \quad if \quad 1 \leq z \leq i \\
\\ \textrm{Hom}_{A_{R_{(i, j, k)}}}(P(j),\frac{P(y)}{P(t)})=0 \quad if \quad i+1 \leq y \leq t-1\\
\\  \textrm{Hom}_{A_{R_{(i, j, k)}}}(P(j),\frac{P(v)}{P(j)})=0 \quad if \quad t+1 \leq v \leq j-1 \\
\\ \textrm{Hom}_{A_{R_{(i, j, k)}}}(P(j),P(u))=0 \quad if \quad j+1 \leq u \leq n\\
\end{array}
\right.$}

\par\bigskip
We conclude that $\textrm{Ext}^{n}_{A_{R_{(i, j, k)}}}(I_{t, j},A_{R_{(i, j, k)}}/I_{t, j})=0$ and that the idempotent ideal $I_{t, j}$ is an homological ideal as a consequence of Proposition \ref{JAP}, item 3.
\end{enumerate}\end{proof}

\addtocounter{Nota}{6}
\begin{Nota}
If the non homological ideal  $I_t$ has the form $I_{t_1,\ldots, t_n}$ the previous Lemma \ref{18} also holds.
\end{Nota}
\addtocounter{lema}{1}
\begin{lema}\label{main}
For $1\leq h \leq i-1$, $1\leq l \leq k-1$ and $1\leq m \leq n-j$ fixed. Every idempotent ideal of the form $I_{z_1,\ldots,z_h, t_1, \ldots ,t_l, y_1, \ldots , y_m }$  of an algebra $A_{R_{(i, j, k)}}$, where $z_a \in [1,i-1],$ $t_b \in [i+1, j-1]$, $y_c\in [j+1,n]$ is not homological.
\end{lema}

\begin{proof}
For $h$, $l$ and $m$ fixed, we start by computing $I_{z_1,\ldots, z_h, t_1, \ldots ,t_l, y_1, \ldots , y_m}$,  
\begin{equation}\label{*}
\begin{split}
I_{z_1,\ldots, z_h, t_1, \ldots ,t_l, y_1,\ldots , y_m}&=tr_{P(z_1)\oplus \ldots \oplus P(z_h)\oplus P(t_1) \oplus \ldots \oplus P(t_l)\oplus P(y_1), \oplus \ldots \oplus P(y_m)}(A_{R_{(i, j, k)}}) \\
&= \underbrace{\displaystyle \sum_{a=1}^{h}tr_{P(z_a)}(A_{R_{(i, j, k)}})}_{(1)} + \underbrace{\displaystyle \sum_{b=1}^{l}tr_{P(t_b)}(A_{R_{(i, j, k)}})}_{(2)} + \underbrace{\displaystyle \sum_{c=1}^{m}tr_{P(y_c)}(A_{R_{(i, j, k)}})}_{(3)}
\end{split}
\end{equation}

The traces $(1)$, $(2)$, $(3)$  can be written as follows:

\begin{equation}
\begin{split}
\displaystyle \sum_{a=1}^{h}tr_{P(z_a)}(A_{R_{(i, j, k)}})&=\bigoplus_{a=1}^{h}P(z_a)\oplus 0 \oplus \cdots \oplus 0, \\
\displaystyle \sum_{b=1}^{l}tr_{P(t_b)}(A_{R_{(i, j, k)}})&= \underbrace{L_{t_1}}_{i- times}\oplus  \bigoplus_{b=1}^{l} P(t_b) \oplus\underbrace{0}_{n-i-l-times},\\
\displaystyle \sum_{c=1}^{m}tr_{P(y_c)}(A_{R_{(i, j, k)}})&= \underbrace{0}_{i- times}\oplus\underbrace{P(y_1)}_{j-i - times} \oplus \bigoplus_{c=1}^{m} P(y_c) \oplus\underbrace{0}_{n-m-j-times}.\\
\end{split}
\end{equation}

Thus, the ideal  $I_{z_1,\ldots, z_h, t_1, \ldots ,t_l, y_1,\ldots , y_m}$ has the following form:
\begin{equation}\label{13}
\bigoplus_{a=1}^{h}P(z_a) \oplus \underbrace{L_{t_1}}_{i-h-times}\oplus \bigoplus_{b=1}^{l}P(t_b)\oplus \underbrace{P(y_1)}_{j-i-l-times} \oplus \bigoplus_{c=1}^{m}P(y_c)\oplus \underbrace{0}_{n-m-j- times}
\end{equation}

In accordance with (\ref{13}) we have that $\frac{P(j)}{P(y_1)}$ is a direct summand of the quotient $A_{R_{(i, j, k)}}/I_{z_1,\ldots, z_h, t_1, \ldots ,t_l, y_1,\ldots , y_m}$ and $L_{t_1}$ has the following projective resolution:

\begin{equation}\label{res19}
\begin{split}
0 \to P(j) \to P(t_1) \to L_{t_1} \to 0.
\end{split}
\end{equation}

Applying the functor $ \textrm{Hom}_{A_{R_{(i, j, k)}}}\left( -,\frac{P(j)}{P(y_1)}\right) $ to the resolution (\ref{res19}) we obtain the following exact sequence

\[
0 \to  \textrm{Hom}_{A_{R_{(i, j, k)}}}\left( P(t_1),\frac{P(j)}{P(y_1)}\right)  \to  \textrm{Hom}_{A_{R_{(i, j, k)}}}\left( P(j),\frac{P(j)}{P(y_1)}\right)  \to 0
\]\par\smallskip

Then $\textrm{Ext}^{1}_{A_{R_{(i, j, k)}}}\left( L_t,\frac{P(j)}{P(y_1)}\right)  \cong k$ and\par\smallskip
\begin{centering}
 $\textrm{Ext}^{1}_{A_{R_{(i, j, k)}}}(I_{z_1,\ldots, z_h, t_1, \ldots ,t_l, y_1,\ldots , y_m},A_{R_{(i, j, k)}}/I_{z_1,\ldots, z_h, t_1, \ldots ,t_l, y_1,\ldots , y_m}) \not= 0$\par\smallskip
\end{centering}

 by Proposition \ref{JAP}, item 3, we conclude that the idempotent ideal $I_{z_1,\ldots, z_h, t_1, \ldots ,t_l, y_1,\ldots , y_m}$ is not an homological ideal.
\end{proof}

\begin{lema}\label{last}
For $1\leq h \leq i-1$, $1\leq l \leq k-1$ and $1\leq m \leq n-j$ fixed. The idempotent ideals $I_{z_1,\ldots,z_h, t_1, \ldots ,t_l}$ and $I_{t_1, \ldots ,t_l, y_1, \ldots , y_m }$  of an algebra $A_{R_{(i, j, k)}}$ where $z_a \in [1,i-1],$  $t_b \in [i+1, j-1],$ $y_c\in [j+1,n]$ are not homological.
\end{lema}

\begin{proof}
It is enough to consider in (\ref{*}) the trace
$\displaystyle \sum_{a=1}^{h}tr_{P(z_a)}(A_{R_{(i, j, k)}})=0$  or the trace  $\displaystyle \sum_{c=1}^{m}tr_{P(y_c)}(A_{R_{(i, j, k)}})=0$.
\end{proof}

\subsection{On the Number of Homological Ideals Associated to Some Nakayama Algebras}
The following results allow us to compute the number of homological and non homological ideals in a bounded algebra $A_{R_{(i, j, k)}}$ by using the integer specialization $e$ of the Brauer configuration $\Gamma_n$ introduced in Example \ref{integerspecialization}.

\addtocounter{teor}{10}
\begin{teor} \label{N.H}
For $n\geq 4$ fixed and $2\leq k \leq n-1$ the number $|\mathbb{NHI}_n^k|$ of non homological ideals of an algebra $A_{R_{(i, j, k)}}$ is given by the identity $|\mathbb{NHI}_n^k|=  w^e(U_k)$.

\end{teor}

\begin{proof}
We note that according to Lemmas \ref{center} and \ref{center2} there are $2^{k-1}-1$ non homological ideals associated only to the vertices inside the relation $R_{(i, j, k)}$, by Lemma \ref{main} there are additional $2^{n-k-1}$ non homological ideals arising from the combination of vertices which are inside and outside of the relation. The theorem follows taking into account the product rule and Example \ref{integerspecialization}.
\end{proof}

\addtocounter{corol}{18}
\begin{corol}\label{H.I}
For $n\geq 4$ fixed and $2\leq k \leq n-1$ the number of homological ideals $|\mathbb{HI}_n^k|$ of an algebra $A_{R_{(i, j, k)}}$ is given by the identity $|\mathbb{HI}_n^k|=2^n-w^{e}(U_k)=3\cdot2^{n-2}+2^{n-k-1}$.

\end{corol}
\begin{proof}
Since there are $2^n$ idempotent ideals in $A_{R_{(i, j, k)}}$ then the result holds as a consequence of Theorem \ref{N.H}.
\end{proof}

The formula obtained in Theorem \ref{N.H} induces the following triangle:
\begin{center}
\textbf{Non homological triangle $\mathbb{NHIT}$}
\vspace{0.5cm}  

\begin{tabular}{|c|c|c|c|c|c|c|c|c|}
\hline
$n/k$ & $2$ & $3$ & $4$ & $5$ & $6$ & $7$ & $8$ &$\cdots$\\
\hline
$3$ & $1$ & - & - & - & - & - & - & -\\
\hline
$4$& $2$ & $3$  & - & - & - & - & - &-\\
\hline
$5$ & $4$ & $6$ & $7$  & - & - & - & - &-\\
\hline
$6$& $8$ & $12$ & $14$ & $15$ & - & - & - &-\\
\hline
$7$ & $16$ & $24$ & $28$ & $30$ & $31$ & - & - &-\\
\hline
$\vdots $& $\vdots$ & $\vdots $ & $\vdots $ & $\vdots $ & $\vdots $ & $\vdots $ & $\vdots$& $\vdots$\\
\hline
\end{tabular}

\end{center}

Entries $|\mathbb{NHI}_n^k|$ of triangle $\mathbb{NHIT}$ can be calculated inductively as follows: we start by defining $|\mathbb{NHI}_n^2|=2^{n-3}$ for all $n\geq 3$. Now, we assume that $|\mathbb{NHI}_n^k|=0$ with $k\leq 1$ and for the sake of clarity we denote the specialization under $e$ of a word $w(U_k)$ of the polygon $U_k$ in the Brauer configuration $\Gamma_n$ as $w^{e}(U_k^n)$ (see Example \ref{integerspecialization}). Then, for $k\geq 3$:

\[w^{e}(U_k)=w^{e}(U_k^n)=(w^{e}(U_{k-1}^n)+ w^{e}(U_{k-1}^{n-1}))-w^{e}(U_{k-2}^{n-1}).\]  

or equivalently,
\[|\mathbb{NHI}_n^k|=(|\mathbb{NHI}_n^{k-1}|+ |\mathbb{NHI}_{n-1}^{k-1}|)-|\mathbb{NHI}_{n-1}^{k-2}|.\]  These arguments prove the following proposition.

\addtocounter{prop}{13}
\begin{prop}
$M(\Gamma_n^e)$ equals the sum of the elements in the $n$-th row of the non homological triangle $\mathbb{NHIT}$ (see Example $\mathrm(\ref{integerspecialization})$).
\end{prop}

\addtocounter{Nota}{5}
\begin{Nota}
The integer sequence generated by $M(\Gamma_n^e)=\displaystyle\sum_{k=1}^{n-1} 2^{n-2}-2^{n-k-1}=\{1,5,17,49,129,321,769,1793,4097,9217, \ldots\}$ is encoded A000337 in the OEIS. Elements of the sequence A000337 also correspond to the sums of the elements of the rows of the Reinhard Zumkeller triangle.
\end{Nota}
\begin{Nota}\label{properties}
The sum of  entries in the diagonals of the non homological triangle is the sequence A274868 in the OEIS, and it is related with the number of set partitions of $[n]$ into exactly four blocks such that all odd elements are in blocks with an odd index, whereas all even elements are in blocks with an even index.
\end{Nota}

Similarly, for the homological ideals Corollary \ref{H.I} induces the following triangle:

\begin{center}
\textbf{ Homological triangle $\mathbb{HIT}$.}
\vspace{0.5cm}  

\begin{tabular}{|c|c|c|c|c|c|c|c|c|}
\hline
$n/k$ & $2$ & $3$ & $4$ & $5$ & $6$ & $7$ & $8$ &$\cdots$\\
\hline
$3$ & $7$ & - & - & - & - & - & - & -\\
\hline
$4$& $14$ & $13$  & - & - & - & - & - &-\\
\hline
$5$ & $28$ & $26$ & $25$  & - & - & - & - &-\\
\hline
$6$& $56$ & $52$ & $50$ & $49$ & - & - & - &-\\
\hline
$7$ & $112$ & $104$ & $100$ & $98$ & $97$ & - & - &-\\
\hline
$\vdots $& $\vdots$ & $\vdots $ & $\vdots $ & $\vdots $ & $\vdots $ & $\vdots $ & $\vdots$& $\vdots$\\

\hline
\end{tabular}
\vspace{0.5cm}

\end{center}

The elements of the homological triangle are closely related with the research of categorification of integer sequences. Particularly, these numbers deal with the work of Ringel and Fahr regarding categorification of Fibonacci numbers. In the next section \ref{categorification}, we reconstruct the partition formula for even-index Fibonacci numbers given in \cite{Fahr1,Fahr3} by using the number of homological ideals of some Nakayama algebras.

\subsection{Categorification of Integer Sequences}\label{categorification}
In this section, we give some relationships between the number of homological ideals of an algebra $A_{R_{(i, j, k)}}$ and the partition formula given by Ringel and Fahr for even-index Fibonacci numbers in \cite{Fahr1}.\par\bigskip
According to Ringel and Fahr \cite{Fahr2} a categorification of a sequence of
numbers means to consider instead of these numbers suitable objects in a category (for instance, representation of quivers) so that the numbers in question occur as invariants of the objects, equality of numbers may be visualized by isomorphisms of objects functional relations by functorial ties.  The notion of this kind of categorification arose from  the use of suitable arrays of numbers to obtain integer partitions of dimensions of indecomposable preprojective modules over the 3-Kronecker algebra (see Figure 3 where it is shown the 3-Kronecker quiver and a piece of the oriented 3-regular tree or universal covering $(T,E,\Omega_{t})$ as described by Ringel and Fahr in \cite{Fahr1}). Firstly they noted that  the vector dimension of these kind of modules consists of even-index Fibonacci numbers (denoted $f_{i}$ and such that $f_{i}=f_{i-1}+f_{i-2}$, for $i\geq2$, $f_{0}=0$, $f_1=1$) then they used results from the universal covering theory developed by Gabriel and his students to identify such Fibonacci numbers with dimensions of representations of the corresponding universal covering.

\begin{figure}[H]
\begin{center}
\begin{tikzpicture}[scale=0.42]
%\draw[step=0.5cm,gray,very thin] (0,0) grid (14,8);
\node["1"] at (2,4)(a){$\circ$};
\node["2"] at (5,4)(b){$\circ$};
\node[inner sep=0pt] at (9,6.5)(9){$\bullet$};
\node[inner sep=0pt] at (12,6.5)(10){$\bullet$};
\node[inner sep=0pt] at (8,5.5)(11){$\bullet$};
\node[inner sep=0pt] at (13,5.5)(12){$\bullet$};
\node[inner sep=0pt] at (10,5.5)(3){$\bullet$};
\node[inner sep=0pt] at (11,5.5)(4){$\bullet$};
\node[inner sep=0pt] at (9,4.5)(1){$\bullet$};
\node[inner sep=0pt] at (12,4.5)(2){$\bullet$};

\node[inner sep=0pt] at (9,3.5)(5){$\bullet$};

\node[inner sep=0pt] at (12,3.5)(6){$\bullet$};

\node[inner sep=0pt] at (10,2.5)(7){$\bullet$};

\node[inner sep=0pt] at (11,2.5)(8){$\bullet$};

\node[inner sep=0pt] at (8,2.5)(13){$\bullet$};

\node[inner sep=0pt] at (13,2.5)(14){$\bullet$};

\node[inner sep=0pt] at (9,1.5)(15){$\bullet$};

\node[inner sep=0pt] at (12,1.5)(16){$\bullet$};

%______________________

\path[->] (b) edge (a);

\path[->](b) edge  [bend left](a);

\path[->](b) edge  [bend right](a);

\path[->] (3) edge (1);

\path[->] (3) edge (4);

\path[->] (3) edge (9);

\path[->] (11) edge (1);

\path[->] (10) edge (4);

\path[->] (2) edge (4);

\path[->] (2) edge (12);

\path[->] (5) edge (1);

\path[->] (2) edge (6);

\path[->] (5) edge (13);

\path[->] (5) edge (7);

\path[->] (14) edge (6);

\path[->] (8) edge (7);
\path[->] (15) edge (7);
\path[->] (8) edge (16);
\path[->] (8) edge (6);
\end{tikzpicture}
\label{covering}
\caption{The 3-Kronecker quiver and an illustration of its corresponding universal covering.}
\end{center}
\end{figure}
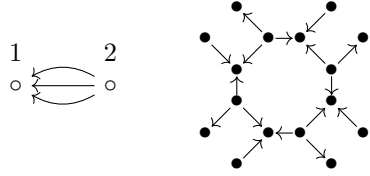
First of all note that the road to a categorification of the Fibonacci numbers has several stops some of them dealing with diagonal (lower) arrays of numbers of the form $D=(d_{i, j})$ with $0\leq j\leq i\leq n$, (columns numbered from right to the left, see Figure 4) for some $n\geq0$ fixed and such that:
\par\bigskip
\addtocounter{equation}{1}
\begin{equation}\label{farray}
\begin{split}
d_{i,i}&=1, \quad\text{for all}\hspace{0.1cm}i\geq0,\\
d_{i,j}&=0 \quad\text{for all}\hspace{0.1cm}j>i\\
d_{2k+i, i-1}&=0, \quad\text{if}\hspace{0.1cm}i\geq 1, \hspace{0.1cm}k\geq 0,\\
d_{2k, 0}&=3d_{2k-1, 1}-d_{2(k-1), 0},\quad k\geq 1,\\
d_{i+1, j-1}&=2d_{i, j}+d_{i, j-2}-d_{i-1, j-1},\quad i, j\geq2.\\
\end{split}
\end{equation}
Besides, if $i \geq 4$ then the following identity (hook rule) holds;

\begin{equation}
\begin{split}
\underset{k=0}{\overset{i-2}{\sum}}d_{i+k, i-k}+d_{2i-2, 0}&=d_{2i-1, 1}.
\end{split}
\end{equation}

Note that to each entry $d_{i, i-j}$ it is possible to assign a weight $w_{i, i- j}$ by using the numbers in the homological triangle $\mathbb{HIT}$ as follows:

\[w_{i, i-j}=
\begin{cases}
|\mathbb{HI}_{2s+2}^k|- 2^{2\cdot s-k+1},  & \hspace{0.2cm}\text{if}\hspace{0.2cm}j\hspace{0.2cm}\text{is even}, i\hspace{0.2cm}\text{is odd}\hspace{0.2cm} \text{and }i\neq j+1, \\
|\mathbb{HI}_{2s+1}^k|- 2^{2\cdot s-k},  & \hspace{0.2cm}\text{if}\hspace{0.2cm}j\hspace{0.2cm}\text{is even}, i\hspace{0.2cm}\text{is even}, \\
3,&\hspace{0.2cm}\text{if}\hspace{0.2cm}i\hspace{0.2cm}\text{odd}, j\hspace{0.2cm}\text{even}   \hspace{0.1cm} \text{and } i=j+1,\\
1,&\hspace{0.2cm}\text{if}\hspace{0.2cm}i=j=2h\hspace{0.2cm}\text{for some}\hspace{0.1cm}h\geq0,\\
0,            &\hspace{0.2cm}\text{if}\hspace{0.2cm}j\hspace{0.2cm}\text{is odd},\hspace{0.2cm}i\neq j.\\
\end{cases}\]

Where $s=\lfloor\frac{i-j}{2}\rfloor$  and $\lfloor x\rfloor$ is the greatest integer number less than $x$. If we consider the multiplication of the entry $d_{i, i-j}$ with its corresponding weight $w_{i, i- j}$ we can define a partition formula for even-index Fibonacci numbers in the following form:
\begin{equation}\label{Fibonacci}
\begin{split}
f_{2i+2}&=\underset{j=0}{\overset{i}{\sum}}(w_{i,i-j})(d_{i, i-j}),
\end{split}
\end{equation}

Finally, we recall that Ringel and Fahr interpreted weights $w_{i, i-j}$ as distances in a 3-regular tree $(T,E)$ (with $T$ a vertex set and $E$ a set of edges) from a fixed point $x_{0}\in T$ to any point $y\in T$. They define sets $T_{r}$ whose points have distance $r$ to $x_{0}$ , in such a case $T_{0}=\{x_{0}\}$, $T_{1}$ are the neighbors of $x_{0}$ and so on (note that $|T_{r}|=3(2^{r-1})$ if $r\geq1$). A given vertex $y$ is said to be even or odd according to this parity \cite{Fahr1}.\par\bigskip
Any vertex $y\in T$ yields a suitable reflection $\sigma_{y}$ on
the set of functions $T\rightarrow \mathbb{Z}$ with finite support, denoted $\mathbb{Z}[T]$, and some reflection products denoted
$\Phi_{0}$ and $\Phi_{1}$ according to the parity of $y$ are introduced in \cite{Fahr1}. Then some maps
$a_{t}:\mathbb{N}_{0}\rightarrow \mathbb{Z}\in\mathbb{Z}[T]$ are defined in such a way that if $a_{0}$ is
the characteristic function of $T_{0}$ then $a_{0}(x)=0$ unless $x=x_{0}$ in which case $a_{0}(x_{0})=1$,\quad and\quad $a_{t}=(\Phi_{0}\Phi_{1})^{t}a_{0}$, for $t\geq0$, with $a_{t}[r]=a_{t}(x)$, for $r\in\mathbb{N}_{0}$ and $x\in T_{r}$, these maps $a_{t}$ give the values $d_{i, j}$ of the array (see Figure 4). The following table is an example of such array with $n=7$. Rows are giving by the values of $t$, $P_{t}$ is a notation for a 3-Kronecker preprojective module with dimension vector $[f_{2t+2}\hspace{0.2cm}f_{2t}]$  (see \cite{Fahr3}). \par\bigskip

According to the present discussion the identity (\ref{Fibonacci}) adopts one of the following forms defined by Ringel and Fahr in \cite{Fahr1}:

 \begin{equation}\label{Fibonacci2}
\begin{split}
f_{4t}&=\underset{r\hspace{0.1cm} odd}{{\sum}}|T_{r}|\cdot a_{t}[r]=3\underset{m\geq1}{{\sum}}2^{2m}\cdot a_{t}[2m+1],\\
f_{4t+2}&=\underset{r\hspace{0.1cm} even}{{\sum}}|T_{r}|\cdot a_{t}[r]=a_{t}[0]+3\underset{m\geq1}{{\sum}}2^{2m-1}\cdot a_{t}[2m].
\end{split}
\end{equation}

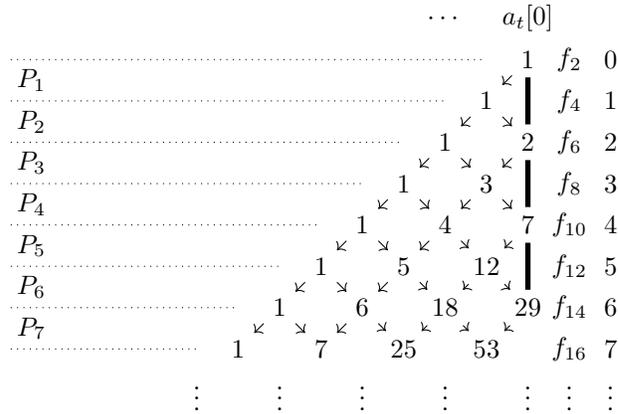
\begin{figure}[H]
\begin{center}
\begin{tikzpicture}[scale=0.55]

%%%%%%0

\node at (10,15)(a){$0$};
\node at (10,14)(a){$1$};
\node at (10,13)(a){$2$};
\node at (10,12)(a){$3$};
\node at (10,11)(a){$4$};
\node at (10,10)(a){$5$};
\node at (10,9)(a){$6$};
\node at (10,8)(a){$7$};
%\node at (10,7)(a){$8$};
%\node at (10,6)(a){$9$};
%\node at (10,5)(a){$10$};
%\node at (10,4)(a){$11$};
%\node at (10,3)(a){$12$};
\node at (10,7)(a){$\vdots$};
%\node at (10,6)(a){$t$};
%%%%%% 0
\node at (9,15)(a){$f_2$};
\node at (9,14)(a){$f_4$};
\node at (9,13)(a){$f_6$};
\node at (9,12)(a){$f_8$};
\node at (9,11)(a){$f_{10}$};
\node at (9,10)(a){$f_{12}$};
\node at (9,9)(a){$f_{14}$};
\node at (9,8)(a){$f_{16}$};
%\node at (9,7)(a){$f_{18}$};
% \node at (9,6)(a){$f_{20}$};
% \node at (9,5)(a){$f_{22}$};
% \node at (9,4)(a){$f_{24}$};
% \node at (9,3)(a){$f_{26}$};
\node at (9,7)(a){$\vdots$};
%\node at (9,6)(a){$f_{2t+2}$};

\node at (8,7)(a){$\vdots$};
\node at (6,7)(a){$\vdots$};
\node at (4,7)(a){$\vdots$};
\node at (2,7)(a){$\vdots$};
\node at (0,7)(a){$\vdots$};
%%%%%%%%%%%%%% Orbit 1
\node at (8,15)(1){$1$};
\node at (8,13)(2){$2$};
\node at (8,11)(3){$7$};
\node at (8,9)(4){$29$};
%\node at (8,7)(5){$130$};
%\node at (8,5)(6){$611$};
%\node at (8,3)(7){$2965$};
%%%%%%%%% Orbit 2 Rigth to Left
\node at (7,14)(8){$1$};
\node at (7,12)(9){$3$};
\node at (7,10)(10){$12$};
\node at (7,8)(11){$53$};
% \node at (7,6)(12){$247$};
% \node at (7,4)(13){$1192$};
%%%%%%%%% Orbit 3 Rigth to Left
\node at (6,13)(14){$1$};
\node at (6,11)(15){$4$};
\node at (6,9)(16){$18$};
%\node at (6,7)(17){$85$};
% \node at (6,5)(18){$414$};
% \node at (6,3)(19){$2062$};

%%%%%%%%% Orbit 4 Rigth to Left
\node at (5,12)(20){$1$};
\node at (5,10)(21){$5$};
\node at (5,8)(22){$25$};
%\node at (5,6)(23){$126$};
%\node at (5,4)(24){$642$};

%%%%%%%%% Orbit 5 Rigth to Left
\node at (4,11)(25){$1$};
\node at (4,9)(26){$6$};
% \node at (5,9){\tiny{(2(6)+18)-5}};
%\node at (4,7)(27){$33$};
%\node at (4,5)(28){$177$};
%\node at (4,3)(29){$943$};

%%%%%%%%% Orbit 6 Rigth to Left
\node at (3,10)(30){$1$};
\node at (3,8)(31){$7$};
% \node at (3,6)(32){$42$};
% \node at (3,4)(33){$239$};

%%%%%%%%% Orbit 7 Rigth to Left
\node at (2,9)(34){$1$};
%\node at (2,7)(35){$8$};
% \node at (2,5)(36){$52$};
% \node at (2,3)(37){$313$};

%%%%%%%%% Orbit 8 Rigth to Left
\node at (1,8)(38){$1$};
%   \node at (1,6)(39){$9$};
% \node at (1,4)(40){$63$};

%%%%%%%%% Orbit 9 Rigth to Left
%\node at (0,7)(41){$1$};
%\node at (0,5)(42){$10$};
%\node at (0,3)(43){$75$};
%%%%%%%%% Orbit 10 Rigth to Left
%\node at (-1,6)(44){$1$};
%\node at (-1,4)(45){$11$};
%%%%%%%%% Orbit 11 Rigth to Left
%\node at (-2,5)(46){$1$};
%\node at (-2,3)(47){$12$};
%%%%%%%%% Orbit 12 Rigth to Left
%\node at (-3,4)(48){$1$};
%%%%%%%%% Orbit 13 Rigth to Left
%\node at (-4,3)(49){$1$};

%%%%%% Last part II
%\node at (-5,6)(50){$$};
%\node at (-3,2)(51){$$};
% \node at (0,7)(52){$$};
% \node at (2,7)(53){$$};
% \node at (4,7)(54){$$};
% \node at (6,7)(55){$$};
% \node at (8,7)(56){$$};

%%%%%% Last part
%\node at (-6,6)(57){};
%\node at (-4,6)(58){};
% \node at (-1,6)(59){};
% \node at (1,6)(60){};
% \node at (3,6)(61){};
% \node at (5,6)(62){};
% \node at (7,6)(63){};

%%%% P(i)

%%%%%% 0
\node at (-4,14.5)(a){$P_1$};
\node at (-4,13.5)(a){$P_2$};
\node at (-4,12.5)(a){$P_3$};
\node at (-4,11.5)(a){$P_4$};
\node at (-4,10.5)(a){$P_5$};
\node at (-4,9.5)(a){$P_6$};
\node at (-4,8.5)(a){$P_7$};
%\node at (-4,7.5)(a){$P_8$};
%\node at (-4,6.5)(a){$P_9$};
% \node at (-4,5.5)(a){$P_{10}$};

%____________________________________
\path[->] (1) edge (8);
\path[->] (8) edge (2);
\path[->] (2) edge (9);
\path[->] (9) edge (3);
\path[->] (3) edge (10);
\path[->] (10) edge (4);
\path[->] (4) edge (11);
%\path[->] (11) edge (5);
% \path[->] (5) edge (12);
% \path[->] (12) edge (6);
% \path[->] (6) edge (13);
% \path[->] (13) edge (7);

%______________________
\path[->] (8) edge (14);
\path[->] (14) edge (9);
\path[->] (9) edge (15);
\path[->] (15) edge (10);
\path[->] (10) edge (16);
\path[->] (16) edge (11);
%\path[->] (11) edge (17);
% \path[->] (17) edge (12);
% \path[->] (12) edge (18);
% \path[->] (18) edge (13);
% \path[->] (13) edge (19);

%______________________
\path[->] (14) edge (20);
\path[->] (20) edge (15);
\path[->] (15) edge (21);
\path[->] (21) edge (16);
\path[->] (16) edge (22);
%\path[->] (22) edge (17);
% \path[->] (17) edge (23);
% \path[->] (23) edge (18);
% \path[->] (18) edge (24);
% \path[->] (24) edge (19);

%______________________
\path[->] (20) edge (25);
\path[->] (25) edge (21);
\path[->] (21) edge (26);
\path[->] (26) edge (22);
%\path[->] (22) edge (27);
% \path[->] (27) edge (23);
% \path[->] (23) edge (28);
% \path[->] (28) edge (24);
% \path[->] (24) edge (29);

%______________________
\path[->] (25) edge (30);
\path[->] (30) edge (26);
\path[->] (26) edge (31);
%\path[->] (31) edge (27);
% \path[->] (27) edge (32);
% \path[->] (32) edge (28);
% \path[->] (28) edge (33);
% \path[->] (33) edge (29);

%______________________
\path[->] (30) edge (34);
\path[->] (34) edge (31);
%\path[->] (31) edge (35);
% \path[->] (35) edge (32);
% \path[->] (32) edge (36);
% \path[->] (36) edge (33);
% \path[->] (33) edge (37);

%______________________
\path[->] (34) edge (38);
%\path[->] (38) edge (35);
% \path[->] (35) edge (39);
% \path[->] (39) edge (36);
% \path[->] (36) edge (40);
% \path[->] (40) edge (37);

%______________________
%\path[->] (38) edge (41);
% \path[->] (41) edge (39);
% \path[->] (39) edge (42);
% \path[->] (42) edge (40);
% \path[->] (40) edge (43);

%______________________
% \path[->] (41) edge (44);
% \path[->] (44) edge (42);
% \path[->] (42) edge (45);
% \path[->] (45) edge (43);

%______________________
% \path[->] (44) edge (46);
% \path[->] (46) edge (45);
% \path[->] (45) edge (47);

%______________________
% \path[->] (46) edge (48);
% \path[->] (48) edge (47);

%______________________
% \path[->] (48) edge (49);

%______________________
\draw[line width=2pt] (1) -- (2);
\draw[line width=2pt] (2) -- (3);
\draw[line width=2pt] (3) -- (4);
%\draw[line width=2pt] (4) -- (5);
% \draw[line width=2pt] (5) -- (6);
% \draw[line width=2pt] (6) -- (7);

%________________ Last part

% \path[->] (50) edge (57);
% \path[->] (51) edge (58);
% \path[->] (52) edge (59);
% \path[->] (53) edge (60);
% \path[->] (54) edge (61);
% \path[->] (55) edge (62);
% \path[->] (56) edge (63);

%________________

% \path[-] (7) edge (56);
% \path[-] (19) edge (56);
% \path[-] (19) edge (55);
% \path[-] (29) edge (54);
% \path[-] (29) edge (55);
% \path[-] (37) edge (53);
% \path[-] (37) edge (54);
% \path[-] (43) edge (52);
% \path[-] (43) edge (53);
% \path[-] (47) edge (51);
% \path[-] (47) edge (52);
% \path[-] (49) edge (50);
% \path[-] (49) edge (51);

\draw[dotted] (-4.5,15) -- (7,15);
\draw[dotted] (-4.5,14) -- (6,14);
\draw[dotted] (-4.5,13) -- (5,13);
\draw[dotted] (-4.5,12) -- (4,12);
\draw[dotted] (-4.5,11) -- (3,11);
\draw[dotted] (-4.5,10) -- (2,10);
\draw[dotted] (-4.5,9) -- (1,9);
\draw[dotted] (-4.5,8) -- (0,8);
%\draw[dotted] (-4.5,7) -- (-1,7);
% \draw[dotted] (-4.5,6) -- (-2,6);
% \draw[dotted] (-4.5,5) -- (-3,5);

\node at (8,16)(a){$a_t[0]$};

\node at (6,16)(a){$\cdots$};

\end{tikzpicture}
\label{triangle1}
\caption{The even-index Fibonacci partition triangle \cite{Fahr3}.}
\end{center}

\end{figure}

For example for $t=3$ and $t=4$, we compute $f_{8}$ and $f_{10}$ as follows;
\begin{equation}
\begin{split}
21&=f_{8}=0+3(3\cdot 2^0)+0+1(3\cdot 2^2),\\
55&=f_{10}=1\cdot 7+0+4(3\cdot 2^1)+0+1(3\cdot 2^{3}).
\end{split}
\label{3}
\end{equation}

sequences $a_{t}[0]=d_{2i,0}$ and $a_{t}[1]=d_{2i+1,1}$ are encoded respectively as A132262 and A110122 in the OEIS. Actually, sequence $a_{t}[0]$ had not been registered in the OEIS before the publication of Ringel and Fahr.\par\bigskip

The following result giving a relationship between  the number of homological ideals and Fibonacci numbers is a direct consequence of identities (\ref{Fibonacci}) and (\ref{Fibonacci2}).

\addtocounter{teor}{4}

\begin{teor}
\begin{equation}
\begin{split}
\underset{j=0}{\overset{2t}{\sum}}(w_{2t,2t-j})(d_{2t, 2t-j})=\underset{r even}{{\sum}}|T_{r}|\cdot a_{t}[r],\quad t\geq0\\
\underset{j=0}{\overset{2t-1}{\sum}}(w_{2t-1,2t-1-j})(d_{2t-1, 2t-1-j})=\underset{r odd}{{\sum}}|T_{r}|\cdot a_{t}[r], \quad t\geq1.
\end{split}
\end{equation}

\end{teor}

%_____________________

\par\bigskip

{Pedro Fernando Fern\'andez Espinosa\hspace{2.0cm}Agust\'{\i}n Moreno Ca\~{n}adas\\
pffernandeze@unal.edu.co\hspace{3.55cm}amorenoca@unal.edu.co\\
Department of Mathematics\hspace{3.15cm}Department of Mathematics\\
Universidad Nacional de Colombia\hspace{2.2cm}Universidad Nacional de Colombia\\
Bogot\'a-Colombia\hspace{4.65cm}Bogot\'a-Colombia}\par\bigskip

\end{document}